\documentclass[12pt]{amsart}

\usepackage{amssymb,amsmath,amsthm}
\usepackage[dvipdfmx]{graphicx}
\usepackage{color}
\usepackage{fancybox}
\usepackage{ascmac}
\usepackage{amsmath}
\usepackage{bm}
\usepackage{graphics}
\newtheorem{Theorem}{Theorem}[section]

\newtheorem{Proposition}[Theorem]{Proposition}

\theoremstyle{definition}
\newtheorem{Remark}[Theorem]{Remark}

\newtheorem{Example}[Theorem]{Example}

\newtheorem{Problem}[Theorem]{Problem}
\newtheorem{Conjecture}[Theorem]{Conjecture}
\newtheorem{Question}[Theorem]{Question}

\def\ZZ{{\mathbb Z}}
\def\RR{{\mathbb R}}

\begin{document}

\title{Lexicographic and reverse lexicographic quadratic Gr\"{o}bner bases of cut ideals}
\author{Ryuichi Sakamoto}
\date{}

\address{Ryuichi Sakamoto,
Department of Mathematical Sciences,
Graduate School of Science and Technology,
Kwansei Gakuin University,
Sanda, Hyogo 669-1337, Japan} 
\email{dpm86391@kwansei.ac.jp}


\begin{abstract}
Hibi conjectured that if a toric ideal has a quadratic Gr\"{o}bner basis,
 then the toric ideal has either a lexicographic or a reverse lexicographic
 quadratic Gr\"{o}bner basis.
 In this paper, we present a cut ideal of a graph that serves as a
 counterexample to this conjecture.
 We also discuss the existence of a quadratic Gr\"{o}bner basis of a cut
 ideal of a cycle.
 Nagel and Petrovi\'{c} claimed that a cut ideal of a cycle has
 a lexicographic quadratic Gr\"{o}bner basis using the results of
 Chifman and Petrovi\'{c}.
 However, we point out that the results of Chifman and Petrovi\'{c} used by Nagel and Petrovi\'{c} are incorrect for cycles of length greater than or equal to 6.
Hence the existence of a quadratic Gr\"{o}bner basis for the cut ideal of a cycle 
(a ring graph) is an open question.
 We also provide a lexicographic quadratic Gr\"{o}bner basis of a cut ideal
 of a cycle of length less than or equal to 7.
\end{abstract}

\maketitle

\section*{Introduction}

A $d \times n$ integer matrix $A=(\bm{a}_{1},\bm{a}_{2},\ldots ,\bm{a}_{n})$ is called a {\em configuration} if there exists a vector $\bm{c}\in \mathbb{R}^d$ such that for all $1\leq{i}\leq{n}$, the inner product $\bm{a}_{i} \cdot \bm{c}$ is equal to $1$.
Let $K$ be a field and let $K[\bm{x}]=K[x_{1},x_{2},\ldots ,x_{n}]$
be a polynomial ring in $n$ variables. For an integer vector $\bm{\alpha}=(\alpha_{1},\alpha_{2},\ldots ,\alpha_{d})\in \mathbb{Z}^d$,
we define the Laurent monomial $\bm{t^{\alpha}}=t_{1}^{\alpha_{1}}t_{2}^{\alpha_{2}}\ldots t_{d}^{\alpha_{d}} \in K[t_1^{\pm 1},t_2^{\pm 1},\ldots,t_d^{\pm 1}]$ and 
$K[A]=K[\bm{t}^{\bm{a}_{1}},\bm{t}^{\bm{a}_{2}},\ldots ,\bm{t}^{\bm{a}_{n}}]$. Let $\pi$ be a homomorphism 
$\pi:K[\bm{x}]\rightarrow K[A]$, where $\pi (x_{i})=\bm{t}^{\bm{a}_{i}}$. The kernel of $\pi$ is called 
the {\em toric ideal} of $A$ and is denoted by $I_{A}$. 
It is known \cite{HOH, Sturm2} that $I_A$ is generated by homogeneous binomials associated to
the kernel of $A$. 
For a configuration $A$, let ${\rm Ker}_{\mathbb{Z}}A=\{\bm{b} \in \mathbb{Z}^n \ | \ A\bm{b}=\bm{0} \}$. 
For each $\bm{b}=(b_1,\ldots,b_n) \in {\rm Ker}_{\mathbb{Z}}A$, 
we define 
$$f_{\bm{b}}=\prod_{b_i>0} {x_i}^{b_i} - \prod_{b_j<0} x_j^{-b_j}
\in K[\bm{x}].$$ 
Then $I_A=\langle f_{\bm{b}} \ | \ \bm{b} \in {\rm Ker}_{\mathbb{Z}}A \rangle$. 
%
%
%
Commutative algebraists are interested in the following properties:
\begin{enumerate}
\item The toric ideal $I_{A}$ is generated by quadratic binomials;
\item The toric ring $K[A]$ is Koszul;
\item There exists a monomial order satisfying that a Gr\"{o}bner basis of $I_{A}$ consists of quadratic binomials.
\end{enumerate}
The implication $(3)\Rightarrow(2)\Rightarrow(1)$ is true, but
both $(1)\Rightarrow(2)$ and  $(2)\Rightarrow(3)$ are false in general
(for example, see \cite{Hibi, OhsugiHibi1}).
Several classes of toric ideals with
lexicographic/reverse lexicographic quadratic Gr\"{o}bner bases are known (for example, see 
\cite{AHT,Dari, OhsugiHibi2, OhsugiHibi3, OhsugiHibi4, Shibata}).
In contrast, in \cite{AokiOhsugiHibiTakemura1, AokiOhsugiHibiTakemura2, OhsugiHibi5},
sorting monomial orders (which are not necessarily lexicographic or reverse lexicographic) are used to construct a quadratic Gr\"{o}bner basis.  
The monomial orders appearing in the theory of toric fiber products \cite{Sulli} constitute another example
that is not necessarily lexicographic or reverse lexicographic.
The following conjecture was presented by Hibi.

\begin{Conjecture}
\label{conj}
Suppose that the toric ideal $I_A$ has a quadratic Gr\"{o}bner {basis}.
Then $I_A$ has either a lexicographic or reverse lexicographic 
quadratic Gr\"{o}bner basis.
\end{Conjecture}

In the present paper, we will present a cut ideal of a graph as a counterexample to this conjecture.

Now, we define the cut ideal of a graph.
Let $G$ be a finite connected simple graph with the vertex set $V(G)=\{1,2,\ldots ,m\}$ and the edge set $E(G)=\{e_{1},e_{2},\ldots ,e_{r}\}$.
Given a subset $C$ of $V(G)$,
we define a vector $\delta_{C}=(d_{1},d_{2},\ldots ,d_{r})\in \{0,1\}^r$ by
\begin{equation*}
d_{i}=\begin{cases}
								1 & \text{$|C\cap e_{i}|=1 \ {(e_i=\{j,k\})}$,}\\
								0 & \text{otherwise.}
							\end{cases}
\end{equation*}
We consider the configuration
\begin{equation*}
A_G=
{\displaystyle
\begin{pmatrix}
\delta_{C_{1}} && \delta_{C_{2}} && \cdots && \delta_{C_{N}}\\
\\
1 && 1 && \cdots && 1\\
\end{pmatrix},
}
\end{equation*}
where $\{\delta_{C} \ | \ C \subset V(G) \} = \{\delta_{C_{1}}, \delta_{C_{2}},\ldots,\delta_{C_{N}}\}$ and $N=2^{m-1}$.
The toric ideal of $A_G$ is called the {\em cut ideal} of $G$ and is denoted by $I_G$
(see \cite{Sturm} for details). 
This definition of the cut ideal is different from that in \cite{Sturm}. However, the two definitions are equivalent. 
In fact, in \cite{Sturm} they say that ``Indeed, the convex hull of the exponent vectors $\phi_G$ is affinely isomorphic 
to ${\rm Cut}^{\square}(G)$.'' Here ${\rm Cut}^{\square}(G)$ is the convex hull of \{$\delta_C \ | \ C\subset V(G)$\}. 
We illustrate this equivalence by an example. 
\begin{Example}
Let $G$ be a cycle of length $4$ with $V(G)=\{1,2,3,4\}, E(G)=\{e_1=\{1,2\}, e_2=\{2,3\}, e_3=\{3,4\}, e_4=\{1,4\}$\}. 
Then $A_G$ is
\begin{equation*}
A_G=
{\displaystyle
\begin{pmatrix}
0 && 1 && 1 && 1 && 0 && 0 && 0 && 1\\
0 && 1 && 0 && 0 && 1 && 1 && 0 && 1\\
0 && 0 && 1 && 0 && 1 && 0 && 1 && 1\\
0 && 0 && 0 && 1 && 0 && 1 && 1 && 1\\
\hline
1 && 1 && 1 && 1 && 1 && 1 && 1 && 1
\end{pmatrix}.
}
\end{equation*}
Here, the $i$-th row of $A_{G}$ is indexed by the edge $e_i$ and
 the $j$-th column of  $A_{G}$ is indexed by the subset $C_j
 \subset \{1,2,3,4\}$, where $C_1=\phi, C_2=\{2\}, C_3=\{2,3\}, C_4=\{2,3,4\}, C_5=\{3\}, C_6=\{3,4\}, C_7=\{4\}, C_8=\{2,4\}$. On the other hand, in \cite{Sturm}, the cut ideal of $G$ is defined as the kernel of homomorphism 
$\phi_G:K[q_{|1234}, q_{2|134},  q_{23|14},$\\ $q_{234|1}, q_{3|124}, q_{34|12}, q_{4|123}, q_{24|13}]\rightarrow K[s_{12}, s_{23}, s_{34}, s_{14}, t_{12}, t_{23}, t_{34}, t_{14}]$ with
$$q_{|1234}\mapsto t_{12}t_{23}t_{34}t_{14} \ \ \ \ \ \ \ q_{2|134}\mapsto s_{12}s_{23}t_{34}t_{14}$$
$$q_{23|14}\mapsto s_{12}t_{23}s_{34}t_{14} \ \ \ \ \ \ \ q_{234|1}\mapsto s_{12}t_{23}t_{34}s_{14}$$
$$q_{3|124}\mapsto t_{12}s_{23}s_{34}t_{14} \ \ \ \ \ \ \ q_{34|12}\mapsto t_{12}s_{23}t_{34}s_{14}$$
$$q_{4|123}\mapsto t_{12}t_{23}s_{34}s_{14} \ \ \ \ \ \ \ q_{24|13}\mapsto s_{12}s_{23}s_{34}s_{14}.$$
So, the cut ideal defined in \cite{Sturm} is the toric ideal of the following configuration $A_{G}'$:
\begin{equation*}
{\displaystyle
\begin{matrix}
s_{12} \\
s_{23}\\
s_{34}\\
s_{14}\\
t_{12}\\
t_{23}\\
t_{34}\\
t_{14}
\end{matrix}
\begin{pmatrix}
0 && 1 && 1 && 1 && 0 && 0 && 0 && 1\\
0 && 1 && 0 && 0 && 1 && 1 && 0 && 1\\
0 && 0 && 1 && 0 && 1 && 0 && 1 && 1\\
0 && 0 && 0 && 1 && 0 && 1 && 1 && 1\\
\hline
1 && 0 && 0 && 0 && 1 && 1 && 1 && 0\\
1 && 0 && 1 && 1 && 0 && 0 && 1 && 0\\
1 && 1 && 0 && 1 && 0 && 1 && 0 && 0\\
1 && 1 && 1 && 0 && 1 && 0 && 0 && 0
\end{pmatrix}
},
\end{equation*}
where $j$-th column is indexed by $j$-th element of $( q_{|1234}, q_{2|134}, q_{23|14}, q_{234|1}, q_{3|124},  q_{34|12},$\\$ q_{4|123}, q_{24|13} )$.
We obtain the following matrix by elementary row operations from $A_{G}'$:
\begin{equation*}
{\displaystyle
\begin{pmatrix}
A_G\\
O
\end{pmatrix}
},
\end{equation*}
where $O$ is a $3\times8$ zero-matrix.
Therefore, ${\rm Ker}_\ZZ A_G={\rm Ker}_\ZZ A_{G}'$.  
\qed
\end{Example}

We introduce important known results on the quadratic Gr\"{o}bner bases of 
cut ideals. 
An edge {\it contraction} for a graph $G$ is an operation that merges two vertices joined by the edge $e$ 
after removing $e$ from $G$.
A graph $H$ is called a {\it minor} of the graph $G$ if $H$ is obtained by deleting some edges and vertices and contracting some edges.  In this paper, $K_{n},K_{m,n}$, and $\mathcal{C}_{n}$ stand for the complete graph with $n$ vertices, the complete bipartite graph
on the vertex set $\{1, 2, \ldots, m\} \cup \{m+1,m+2,\ldots, m+n\}$ and the cycle of length $n$, respectively.

\begin{Proposition}[\cite{Engst}]
\label{quadgene}
Let $G$ be a graph.
Then $I_G$ is generated by quadratic binomials
if and only if $G$ is free of $K_4$ minors.
\end{Proposition}

\begin{Proposition}[\cite{Shibata}]
Let $G$ be a graph.
Then $K[A_G]$ is strongly Koszul if and only if $G$
is free of $(K_4, \mathcal{C}_5)$ minors.
In addition, if $K[A_G]$ is strongly Koszul, then $I_G$ has a 
quadratic Gr\"{o}bner basis.
\end{Proposition}

Nagel and Petrovi\'{c} \cite[Proposition 3.2]{Petro2} claimed that
if $G$ is a cycle, then  $I_{G}$ has a (lexicographic) quadratic Gr\"{o}bner basis.
However,  \cite[Propositions 2 and 3]{Petro1}, which are used in the proof of
\cite[Proposition 3.2]{Petro2}, contain some errors.
We will explain this in Section 2.
In contrast, the following problem is open.

\begin{Problem}
Classify the graphs whose cut ideals have a 
quadratic Gr\"{o}bner basis.
\end{Problem}

This paper comprises Sections $1$ and $2$. 
In Section 1, we show some results concerning the existence of a
lexicographic/reverse lexicographic quadratic Gr\"{o}bner basis
of cut ideals. 
Then, we give a graph whose cut ideal is a counterexample to Conjecture~\ref{conj}.
In Section 2, we study the cut ideal of a cycle.
First, we point out an error in the lexicographic quadratic Gr\"{o}bner basis of cut ideals of cycles given in
\cite[Proposition 3]{Petro1} (and introduced in \cite{Petro2}).
Finally, we construct a lexicographic quadratic Gr\"{o}bner basis of the cut ideal of a cycle of length $\le 7$.

\section{Lexicographic and reverse lexicographic Gr\"{o}bner bases}

In this section, we present necessary conditions for cut ideals to have a lexicographic/reverse lexicographic quadratic Gr\"{o}bner basis.
Using these results, we present a graph whose cut ideal is a counterexample to Conjecture~\ref{conj}.

First, we study reverse lexicographic quadratic Gr\"{o}bner bases of cut ideals.
The following was proved in \cite[Theorem~1.3]{Sturm}.

\begin{Proposition}
\label{compressed}
Let $G$ be a graph.
Then 
{the graph $G$ is free of $K_5$ minors and has no induced cycles of length $\ge 5$ 
if and only if
there exists a reverse lexicographic order such that
the initial ideal of $I_G$ is squarefree.}
\end{Proposition}

Using that fact that $A_G$ is a $(0,1)$ matrix and Proposition~\ref{compressed}, we are able to prove the following.

\begin{Proposition}
\label{inducedcycle}
Suppose that a graph $G$ has an induced cycle of length $\ge 5$.
Then $I_G$ has no reverse lexicographic quadratic Gr\"{o}bner bases.
\end{Proposition}

\begin{proof}
Suppose that $I_G$ has a reverse lexicographic quadratic
reduced Gr\"{o}bner basis ${\mathcal G}$.
Any toric ideal is prime in general, and hence ${\mathcal G}$ consists of irreducible binomials.
Since $A_G$ is a configuration, ${\mathcal G}$ consists of homogeneous binomials.
Moreover, since $A_G$ is a $(0,1)$ matrix, there exist no nonzero binomials
of the form $x_i^2 - x_j x_k$ in $I_G$.
In fact, if $x_{i}^{2}-x_jx_k\neq 0$ belongs to $I_G$, then 
$2\delta_{C_{i}}=\delta_{C_{j}}+\delta_{C_{k}}$. 
However, this is impossible since $\delta_{C_{i}}, \delta_{C_{j}}, \delta_{C_{k}}$ are ($0,1$)-vectors. 
It therefore follows that the initial ideal is generated
by squarefree monomials.
By proposition~\ref{compressed}, $G$ has no induced cycle of length $\ge 5$.
\end{proof}

Second, we study the lexicographic quadratic Gr\"{o}bner bases of cut ideals.
Let $G$ be a complete bipartite graph $K_{2,3}$, as shown in Fig.~\ref{K23}.
\begin{figure}[h]
\begin{center}
\includegraphics[width=35mm, pagebox=cropbox, clip]{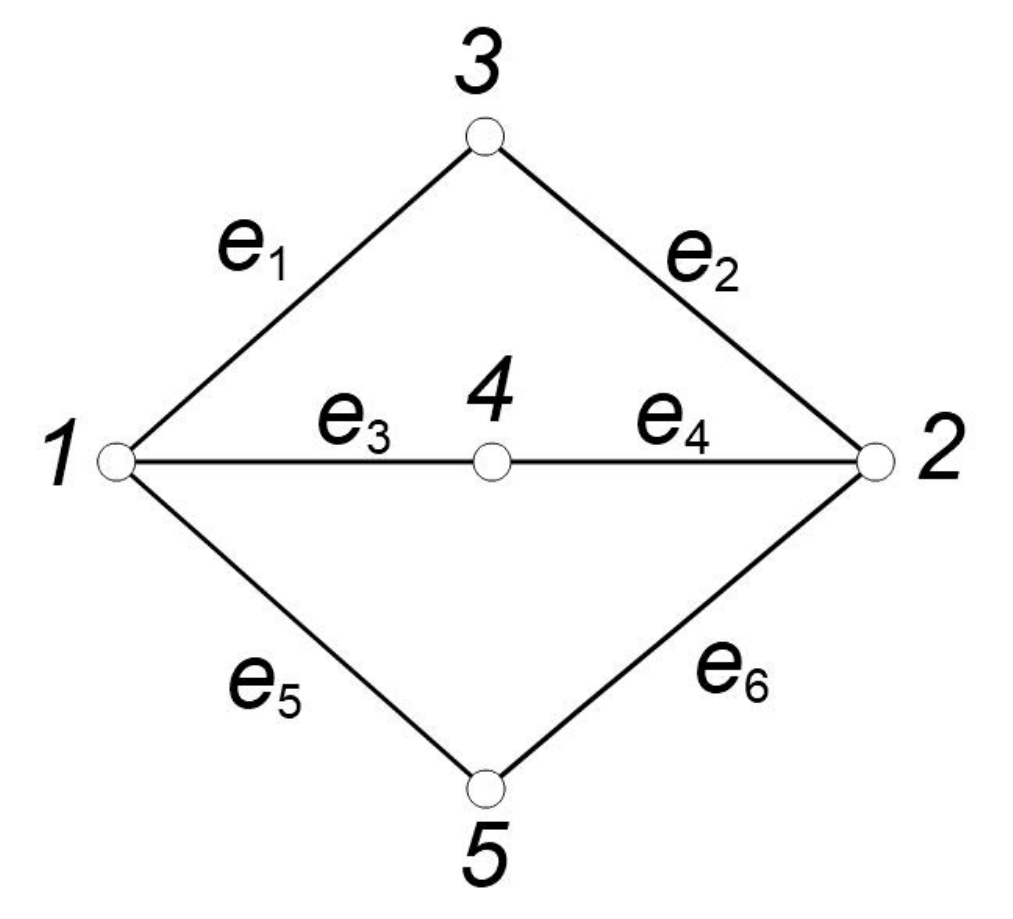}
\caption{ Complete bipartite graph $K_{2,3}$.}
\label{K23}
\end{center}
\end{figure}
The configuration $A_G$ is 
$$A_{G}=
\left(
\begin{array}{cccccccccccccccc}
0 & 0 & 0 & 0 & 1 & 1 & 1 & 1 & 0 & 0 & 0 & 0 & 1 & 1 & 1 & 1\\
0 & 0 & 0 & 0 & 0 & 0 & 0 & 0 & 1 & 1 & 1 & 1 & 1 & 1 & 1 & 1\\
\hline
0 & 0 & 1 & 1 & 1 & 1 & 0 & 0 & 1 & 1 & 0 & 0 & 0 & 0 & 1 & 1\\
0 & 0 & 1 & 1 & 0 & 0 & 1 & 1 & 0 & 0 & 1 & 1 & 0 & 0 & 1 & 1\\
\hline
0 & 1 & 0 & 1 & 1 & 0 & 1 & 0 & 1 & 0 & 1 & 0 & 0 & 1 & 0 & 1\\
0 & 1 & 0 & 1 & 0 & 1 & 0 & 1 & 0 & 1 & 0 & 1 & 0 & 1 & 0 & 1\\
\hline
1 & 1 & 1 & 1 & 1 & 1 & 1 & 1 & 1 & 1 & 1 & 1 & 1 & 1 & 1 & 1
\end{array}
\right).
$$
{Here, the $i$-th row of $A_{G}$ is indexed by the edge $e_i$ and
 the $j$-th column of  $A_{G}$ is indexed by the subset $C_j
 \subset \{1,2,3,4,5\}$,
 where $C_1 = \emptyset, C_2 = \{5\}, C_3= \{4\}, C_4=\{4,5\}, C_5=\{2,3,4,5\}, 
C_6=\{2,3,4\}, C_7=\{2,3,5\}, C_8=\{2,3\}, C_9=\{2,4,5\}, C_{10}=\{2,4\}, 
C_{11}=\{2,5\}, C_{12}=\{2\}, C_{13}=\{3\}, C_{14}=\{3,5\}, C_{15}=\{3,4\}, C_{16}=\{3,4,5\}$. 
The configuration $A_{G}$ has a symmetry group, called switching in \cite{Deza}, as follows.\\
Given subsets $A, B\subset \{1,2,3,4,5\}$, let $A\triangle B$ denote the symmetric difference $(A\cup B)\setminus(A\cap B)$ 
of them. From the general theory of cuts, for any $C, C'\subset \{1,2,3,4,5\}, \delta_C+\delta_{C'}=\delta_{C\triangle C'}$ 
in $\mathbb{F}_{2}^6$. Hence each $C\subset \{1,2,3,4,5\}$ gives a permutation $\psi_C$ on 
$(\delta_{C_1}, \cdots, \delta_{C_{16}})$ defined by 
$$\psi_C(\delta_{C_1}, \cdots, \delta_{C_{16}})=(\delta_{C_{i_1}}, \cdots, \delta_{C_{i_{16}}}),$$
where $\delta_{C_k}+\delta_{C}=\delta_{C_{i_k}}$ in $\mathbb{F}_{2}^6$. The permutation $\psi_C$ naturally induces 
an action on $K[\bm{x}]$ by $\psi_C(x_k)=x_{i_k}$. Since
\begin{equation*}
{\displaystyle
\begin{pmatrix}
\delta_{C_{1}}+\delta_C && \cdots && \delta_{C_{16}}+\delta_C\\
\\
1 && \cdots && 1\\
\end{pmatrix}
}
\end{equation*}
is obtained by elementary row operations from
\begin{equation*}
{\displaystyle
\begin{pmatrix}
\delta_{C_{1}} && \cdots && \delta_{C_{16}}\\
\\
1 && \cdots && 1\\
\end{pmatrix},
}
\end{equation*}
their kernels are the same. Hence we have $\psi_C(I_G)=I_G$. 
We show that $I_{G}$ has no lexicographic quadratic Gr\"{o}bner bases by using these symmetries. 

\begin{Proposition}
\label{k23}
The cut ideal of  the complete bipartite graph  $K_{2,3}$ is generated by quadratic binomials and
 has no lexicographic quadratic Gr\"{o}bner bases.
\end{Proposition}

\begin{proof}
Since $K_{2,3}$ is free of $K_4$ minors,
$I_{K_{2,3}}$ is generated by quadratic binomials according to Proposition~\ref{quadgene}.
Let $<$ be a lexicographic order on $K[{\bf x}]$.
Suppose that the initial ideal of $I_{K_{2,3}}$ with respect to $<$ is quadratic.
Let ${\mathcal M}$ be the set of all monomials in $K[{\bf x}]$
and let $$S =
\{u \in {\mathcal M} \ | \ \pi (u) = t_1 t_2 t_3 t_4 t_5 t_6 t_7^2 \}
.$$
Then we have
$$
S=
\{
x_{1} x_{16}, x_{2} x_{15},  x_{3} x_{14},  x_{4} x_{13}, 
 x_{5} x_{12},  x_{6} x_{11},  x_{7} x_{10},  x_{8} x_{9}
\}.
$$
For each element $x_ix_{17-i}\in S$, 
$\psi_{C_{i}}(x_ix_{17-i})=x_1x_{16}$ for $i=2,\ldots ,8$. 
(For example, 
$\psi_{C_{2}}(x_2x_{15})=x_1x_{16} \  {\rm for} \ C_2=\{5\}$ 
since $\delta_{C_2}+\delta_{C_2}=\delta_{C_1}$ and $\delta_{C_{15}}+\delta_{C_2}=\delta_{C_{16}}$ in $\mathbb{F}_{2}^6$.) 
Hence we may assume that $x_{1} x_{16}$ is the smallest monomial in $S$
with respect to $<$.
It then follows that $x_{1} x_{16} \notin {\rm in}_<(I_{K_{2,3}})$.
We now consider the following 8 cubic binomials of $I_{K_{2,3}}$:
$$
\begin{array}{ccccc}
f_1 &=& x_{6} x_{7} x_{9} &- & x_{1} x_{5} x_{16},\\
f_2 &=& x_{5} x_{8} x_{10} &- & x_{1} x_{6} x_{16},\\
f_3 &=& x_{5} x_{8} x_{11} &- & x_{1} x_{7} x_{16},\\
f_4 &=& x_{6} x_{7} x_{12} &- & x_{1} x_{8} x_{16},\\
f_5 &=& x_{5} x_{10} x_{11} &- & x_{1} x_{9} x_{16},\\
f_6 &=& x_{6} x_{9} x_{12} &- & x_{1} x_{10} x_{16},\\
f_7 &=& x_{7} x_{9} x_{12} &- & x_{1} x_{11} x_{16},\\
f_8 &=& x_{8} x_{10} x_{11} &- & x_{1} x_{12} x_{16}.
\end{array}
$$
Suppose that there exists a nonzero binomial $x_1 x_i - x_j x_k
\in I_{K_{2,3}}$ with  $i \in \{5,\ldots, 12\}$.
Then we have $\delta_{C_i}=\delta_{C_{j}}+\delta_{C_{k}}$.
Since $\delta_{C_i}$ contains exactly 3 ones, so does $\delta_{C_{j}}+\delta_{C_{k}}$.
It then follows that one of $C_j$ and $C_k$ is $C_1$ and hence $x_1x_i - x_j x_k=0$.
Similarly, suppose that there exists a nonzero binomial 
$x_i x_{16}- x_j x_k
\in I_{K_{2,3}}$ with  $i \in \{5,\ldots, 12\}$.
Then we have $\delta_{C_i} +\delta_{C_{16}} =\delta_{C_{j}}+\delta_{C_{k}}$.
Since the sum of the components of $\delta_{C_i} +\delta_{C_{16}}$ is $9$, it follows that one of $C_j$ and $C_k$ is $C_{16}$ and hence $x_i x_{16} - x_j x_k=0$.
%
Thus $x_{1} x_{16}, x_1 x_i , x_i x_{16} \notin  {\rm in}_<(I_{K_{2,3}})$
for each $i \in \{5,\ldots, 12\}$.
{If $x_1 x_i x_{16}$ belongs to ${\rm in}_< (I_{K_{2,3}})$
for some $i \in \{5, \dots, 12\}$, then the cubic monomial
$x_1 x_i x_{16}$ belongs to the minimal set of monomial generators
of ${\rm in}_< (I_{K_{2,3}})$.
This contradicts the hypothesis that ${\rm in}_< (I_{K_{2,3}})$ is
generated by quadratic monomials.
Hence each $x_1 x_i x_{16}$ does not belong to ${\rm in}_< (I_{K_{2,3}})$. }
Thus the initial monomial of each cubic binomial $f_i$ $(1 \le i \le 8)$ above
is the first monomial.
Let $R=K[x_1, x_5,x_6, \ldots, x_{12}, x_{16}]$. 
Note that each $f_i$ belongs to $R$. 
Let $x_k \ (k\in\{1,5,6,\ldots,12,16\})$ be the greatest variable in $R$ with respect to the lexicographic order. 
Then $x_k$ appears in the second monomial of $f_j$ for some $j$. 
Since $<$ is a lexicographic order, the initial monomial of $f_j$ is the second monomial, a contradiction. 
\end{proof}

\begin{Remark}
Shibata \cite{Shibata} showed that the cut ideal of  the complete bipartite graph $K_{2,m}$
has a quadratic Gr\"{o}bner basis with respect to a reverse lexicographic order.
\end{Remark}

Let $A=(\bm{a}_{1},\bm{a}_{2},\ldots ,\bm{a}_{n})$ be a $d \times n$ configuration and
let $B=(\bm{a}_{i_1},\bm{a}_{i_2},\ldots ,\bm{a}_{i_m})$ be a submatrix of $A$.
Then $K[B]$ is called a {\em combinatorial pure subring} of $K[A]$
if there exists a vector $\bm{c}\in \mathbb{R}^d$ such that
$$
\bm{a}_i \cdot \bm{c}
\left\{
\begin{array}{cc}
=1 & i \in \{i_1,i_2,\ldots, i_m\}, \\
\\
<1 & \mbox{otherwise.}
\end{array}
\right.
$$
That is, $K[B]$ is a combinatorial pure subring of $K[A]$
if and only if there exists a face $F$ of the convex hull of $A$ such that
$\{\bm{a}_{1},\bm{a}_{2},\ldots ,\bm{a}_{n}\} \cap F =
\{\bm{a}_{i_1},\bm{a}_{i_2},\ldots ,\bm{a}_{i_m}\}$.
It is known that a combinatorial pure subring $K[B]$ inherits numerous properties
of $K[A]$
(see \cite{HOH}).
In particular, we have the following:

\begin{Proposition}
\label{cp}
Suppose that $K[B]$ is a combinatorial pure subring of $K[A]$.
If $I_A$ has a lexicographic quadratic Gr\"{o}bner basis,
then so does $I_B$.
\end{Proposition}

Suppose that a graph $H$ is obtained by an edge contraction from a graph $G$;
then it is known from \cite[Lemma 3.2 (2)]{Sturm} that $K[A_H]$ is a combinatorial pure subring of $K[A_G]$.
Thus we have the following from Propositions~\ref{k23} and \ref{cp}.

\begin{Proposition}
\label{k23contraction}
Let $G$ be a graph.
Suppose that $K_{2,3}$ is obtained by a sequence of contractions from $G$.
Then $I_G$ has no lexicographic quadratic Gr\"{o}bner bases.
\end{Proposition}

Let $G$ be a graph with $6$ vertices and $7$ edges, as shown in Fig.~\ref{6-1}.
\begin{figure}[h]
\begin{center}
\includegraphics[width=35mm, pagebox=cropbox, clip]{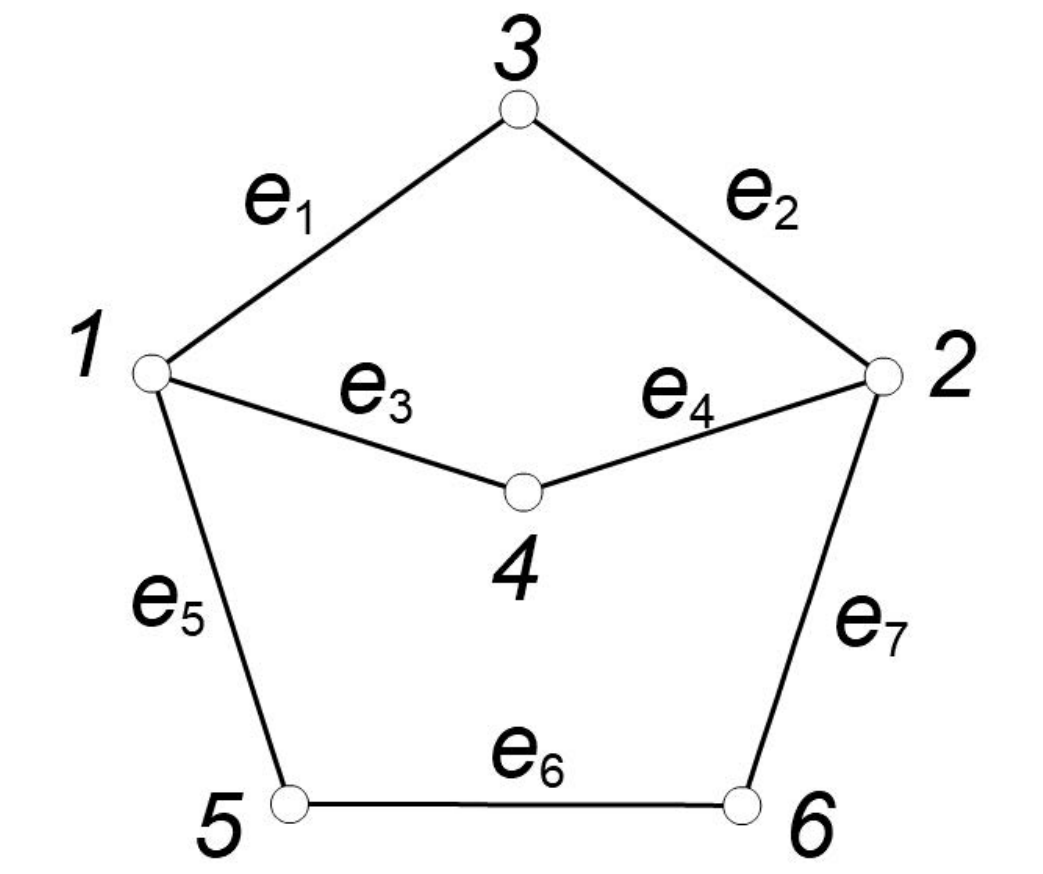}
\caption{ A counterexample to Conjecture ~\ref{conj}.}
\label{6-1}
\end{center}
\end{figure}
\noindent
Then the configuration $A_{G}$ is

\begin{center}
$\left( \begin{smallmatrix}
0 & 0 & 0 & 0 & 0 & 0 & 0 & 0 & 0 & 0 & 0 & 0 & 0 & 0 & 0 & 0 & 1 & 1 & 1 & 1 & 1 & 1 & 1 & 1 & 1 & 1 & 1 & 1 & 1 & 1 & 1 & 1\\
0 & 0 & 0 & 0 & 0 & 0 & 0 & 0 & 1 & 1 & 1 & 1 & 1 & 1 & 1 & 1 & 1 & 1 & 1 & 1 & 0 & 0 & 0 & 0 & 0 & 0 & 0 & 0 & 1 & 1 & 1 & 1\\
0 & 0 & 0 & 0 & 1 & 1 & 1 & 1 & 1 & 1 & 1 & 1 & 0 & 0 & 0 & 0 & 0 & 0 & 0 & 0 & 1 & 1 & 1 & 1 & 0 & 0 & 0 & 0 & 1 & 1 & 1 & 1\\
0 & 0 & 0 & 0 & 1 & 1 & 1 & 1 & 0 & 0 & 0 & 0 & 1 & 1 & 1 & 1 & 0 & 0 & 0 & 0 & 0 & 0 & 0 & 0 & 1 & 1 & 1 & 1 & 1 & 1 & 1 & 1\\
0 & 0 & 1 & 1 & 0 & 0 & 1 & 1 & 1 & 1 & 0 & 0 & 1 & 1 & 0 & 0 & 0 & 0 & 1 & 1 & 1 & 1 & 0 & 0 & 1 & 1 & 0 & 0 & 0 & 0 & 1 & 1\\
0 & 1 & 1 & 0 & 0 & 1 & 1 & 0 & 1 & 0 & 0 & 1 & 1 & 0 & 0 & 1 & 0 & 1 & 1 & 0 & 1 & 0 & 0 & 1 & 1 & 0 & 0 & 1 & 0 & 1 & 1 & 0\\
0 & 1 & 0 & 1 & 0 & 1 & 0 & 1 & 1 & 0 & 1 & 0 & 1 & 0 & 1 & 0 & 0 & 1 & 0 & 1 & 1 & 0 & 1 & 0 & 1 & 0 & 1 & 0 & 0 & 1 & 0 & 1\\
1 & 1 & 1 & 1 & 1 & 1 & 1 & 1 & 1 & 1 & 1 & 1 & 1 & 1 & 1 & 1 & 1 & 1 & 1 & 1 & 1 & 1 & 1 & 1 & 1 & 1 & 1 & 1 & 1 & 1 & 1 & 1
\end{smallmatrix} \right)$.
\end{center}
\vspace{5truept}
{Here, the $i$-th row of $A_{G}$ is indexed by the edge $e_i$ and
 the $j$-th column of  $A_{G}$ is indexed by the subset $C_j
 \subset \{1,2,3,4,5,6\}$,
where $C_1 = \emptyset, C_2 = \{6\}, C_3= \{5\}, C_4=\{5,6\}, C_5=\{4\}, 
C_6=\{4,6\}, C_7=\{4,5\}, C_8=\{4,5,6\}, C_9=\{2,4,5\}, C_{10}=\{2,4,5,6\}, 
C_{11}=\{2,4\}, C_{12}=\{2,4,6\}, C_{13}=\{2,5\}, C_{14}=\{2,5,6\}, C_{15}=\{2\}, 
C_{16}=\{2,6\}, C_{17}=\{3\}, C_{18}=\{3,6\}, C_{19}=\{3,5\}, C_{20}=\{3,5,6\}, 
C_{21}=\{2,3,4,5\}, C_{22}=\{2,3,4,5,6\}, C_{23}=\{2,3,4\}, C_{24}=\{2,3,4,6\}, C_{25}=\{2,3,5\}, 
C_{26}=\{2,3,5,6\}, C_{27}=\{2,3\}, C_{28}=\{2,3,6\}, C_{29}=\{3,4\}, C_{30}=\{3,4,6\}, C_{31}=\{3,4,5\}, C_{32}=\{3,4,5,6\}. $}
The configuration $A_{G}$ contains six combinatorial pure subrings which are isomorphic to $A_{K_{2,3}}$.
By considering weight vectors such that the reduced Gr\"{o}bner basis of $I_{K_{2,3}}$ 
is quadratic, we found a weight vector $\bm{w} \in \RR^{32}$ such that the reduced Gr\"{o}bner basis of $I_G$ 
is also quadratic.
Let $\bm{w}
=(25, 24, 24, 45, 46, 44, 37, 37, 47, 47, 63, 107, 47, 25, 24,$\\ $46, 36,
 33,20, 26, 102, 87, 80,103, 92, 35, 25, 26, 53, 37, 22, 27)$.
The following Gr\"{o}bner basis of $I_G$ with respect to $\bm{w}$ is quadratic:\\
$\{-x_{20}x_{31}+x_{19}x_{32},-x_{15}x_{3}+x_{14}x_{2},x_{28}x_{20}-x_{27}x_{19},-x_{27}x_{31}+x_{28}x_{32},x_{18}x_{31}-x_{30}x_{19},\\
x_{3}x_{32}-x_{8}x_{19},x_{3}x_{31}-x_{7}x_{19},x_{2}x_{19}-x_{18}x_{3},-x_{15}x_{19}+x_{18}x_{14},-x_{26}x_{15}+x_{27}x_{14},\\
x_{27}x_{3}-x_{26}x_{2},x_{1}x_{19}-x_{17}x_{3},-x_{17}x_{2}+x_{1}x_{18},x_{2}x_{31}-x_{30}x_{3},-x_{15}x_{31}+x_{30}x_{14},\\
-x_{30}x_{20}+x_{18}x_{32},x_{7}x_{27}-x_{28}x_{8},x_{7}x_{20}-x_{3}x_{32},-x_{8}x_{31}+x_{7}x_{32},x_{2}x_{31}-x_{6}x_{19},\\
x_{3}x_{20}-x_{4}x_{19},-x_{1}x_{31}+x_{5}x_{19},x_{4}x_{31}-x_{3}x_{32},x_{27}x_{19}-x_{18}x_{26},-x_{6}x_{3}+x_{7}x_{2},\\
-x_{7}x_{15}+x_{6}x_{14},-x_{6}x_{20}+x_{2}x_{32},x_{2}x_{32}-x_{8}x_{18},x_{2}x_{31}-x_{7}x_{18},-x_{5}x_{3}+x_{1}x_{7},\\
-x_{5}x_{2}+x_{1}x_{6},x_{27}x_{3}-x_{4}x_{28},x_{28}x_{15}-x_{16}x_{27},-x_{8}x_{20}+x_{4}x_{32},x_{27}x_{31}-x_{30}x_{26},\\
x_{1}x_{32}-x_{10}x_{27},x_{2}x_{32}-x_{9}x_{27},x_{16}x_{20}-x_{15}x_{19},x_{5}x_{20}-x_{1}x_{32},-x_{15}x_{3}+x_{13}x_{1},\\
-x_{10}x_{2}+x_{9}x_{1},-x_{15}x_{31}+x_{16}x_{32},x_{17}x_{31}-x_{29}x_{19},x_{1}x_{31}-x_{10}x_{28},x_{2}x_{31}-x_{9}x_{28},\\
-x_{1}x_{32}+x_{17}x_{8},x_{1}x_{31}-x_{17}x_{7},x_{6}x_{32}-x_{8}x_{30},x_{6}x_{31}-x_{7}x_{30},x_{1}x_{31}-x_{29}x_{3},\\
-x_{29}x_{2}+x_{1}x_{30},x_{30}x_{2}-x_{6}x_{18},x_{2}x_{20}-x_{4}x_{18},-x_{29}x_{20}+x_{17}x_{32},-x_{6}x_{26}+x_{7}x_{27},\\
x_{5}x_{18}-x_{1}x_{30},-x_{1}x_{30}+x_{17}x_{6},x_{28}x_{14}-x_{16}x_{26},x_{1}x_{20}-x_{17}x_{4},x_{2}x_{32}-x_{4}x_{30},\\
x_{3}x_{32}-x_{9}x_{26},x_{29}x_{1}-x_{17}x_{5},x_{8}x_{3}-x_{4}x_{7},-x_{11}x_{19}+x_{10}x_{18},x_{15}x_{19}-x_{13}x_{17},\\
x_{10}x_{18}-x_{9}x_{17},-x_{7}x_{15}+x_{16}x_{8},-x_{11}x_{31}+x_{10}x_{30},-x_{29}x_{18}+x_{17}x_{30},-x_{11}x_{3}+x_{10}x_{2},\\
-x_{10}x_{15}+x_{11}x_{14},-x_{1}x_{30}+x_{11}x_{28},x_{8}x_{2}-x_{4}x_{6},x_{5}x_{32}-x_{29}x_{8},x_{5}x_{31}-x_{29}x_{7},\\
x_{15}x_{3}-x_{16}x_{4},x_{1}x_{8}-x_{5}x_{4},x_{7}x_{15}-x_{13}x_{5},-x_{10}x_{6}+x_{9}x_{5},x_{9}x_{14}-x_{13}x_{10},\\
x_{5}x_{30}-x_{29}x_{6},-x_{1}x_{32}+x_{29}x_{4},-x_{1}x_{32}+x_{11}x_{26},x_{15}x_{31}-x_{13}x_{29},-x_{10}x_{30}+x_{9}x_{29},\\
-x_{11}x_{7}+x_{10}x_{6},-x_{23}x_{15}+x_{11}x_{27},-x_{1}x_{32}+x_{23}x_{14},x_{9}x_{15}-x_{11}x_{13},-x_{1}x_{32}+x_{22}x_{15},\\
x_{23}x_{3}-x_{22}x_{2},x_{22}x_{14}-x_{10}x_{26},-x_{23}x_{26}+x_{22}x_{27},-x_{25}x_{15}+x_{13}x_{27},-x_{25}x_{14}+x_{13}x_{26},\\
-x_{27}x_{3}+x_{25}x_{1},x_{23}x_{19}-x_{22}x_{18},-x_{23}x_{31}+x_{22}x_{30},-x_{1}x_{30}+x_{23}x_{16},x_{2}x_{32}-x_{21}x_{15},\\
x_{24}x_{15}-x_{1}x_{30},-x_{2}x_{32}+x_{23}x_{13},-x_{3}x_{32}+x_{21}x_{14},x_{23}x_{3}-x_{21}x_{1},-x_{24}x_{27}+x_{23}x_{28}\\
,x_{27}x_{19}-x_{25}x_{17},x_{1}x_{31}-x_{24}x_{14},x_{24}x_{20}-x_{23}x_{19},x_{23}x_{31}-x_{24}x_{32},-x_{23}x_{7}+x_{22}x_{6},\\
-x_{12}x_{15}+x_{11}x_{16},-x_{12}x_{27}+x_{1}x_{30},-x_{12}x_{14}+x_{10}x_{16},x_{1}x_{31}-x_{22}x_{16},x_{12}x_{20}-x_{10}x_{18},\\
-x_{12}x_{32}+x_{10}x_{30},-x_{3}x_{32}+x_{22}x_{13},-x_{24}x_{26}+x_{22}x_{28},x_{13}x_{28}-x_{25}x_{16},-x_{7}x_{27}+x_{25}x_{5},\\
x_{23}x_{19}-x_{21}x_{17},x_{3}x_{32}-x_{25}x_{10},-x_{24}x_{8}+x_{23}x_{7},x_{1}x_{31}-x_{12}x_{26},-x_{12}x_{8}+x_{10}x_{6},\\
x_{27}x_{31}-x_{25}x_{29},-x_{23}x_{3}+x_{24}x_{4},x_{2}x_{31}-x_{21}x_{16},-x_{23}x_{7}+x_{21}x_{5},-x_{12}x_{28}+x_{24}x_{16},\\
-x_{25}x_{9}+x_{21}x_{13},-x_{21}x_{10}+x_{22}x_{9},x_{2}x_{31}-x_{24}x_{13},-x_{23}x_{10}+x_{22}x_{11},x_{10}x_{2}-x_{12}x_{4},\\
x_{9}x_{16}-x_{12}x_{13},-x_{23}x_{31}+x_{21}x_{29},x_{2}x_{32}-x_{25}x_{11},x_{23}x_{9}-x_{21}x_{11},x_{21}x_{27}-x_{25}x_{23},\\
x_{21}x_{26}-x_{25}x_{22},x_{24}x_{11}-x_{12}x_{23},x_{24}x_{10}-x_{12}x_{22},x_{21}x_{28}-x_{24}x_{25},x_{2}x_{31}-x_{12}x_{25},\\
x_{24}x_{9}-x_{12}x_{21}\}$. \vspace{5truept} \\
For the sake of reliability, we computed this using several different
 software packages 
(CoCoA \cite{CoCoA}, Risa/Asir \cite{Asir}, and so on).
The code for the computation is available~in 
\begin{center}
{\tt https:\slash\slash sci-tech.ksc.kwansei.ac.jp/\textasciitilde hohsugi/R\_Sakamoto/code\_cutideal}
\end{center}
For example, if we input 
\begin{screen}
{\tt
M:=MakeTermOrd(mat([[25,24,24,45,46,44,37,37,47,47,63,107,47,25,24,\\
46,36,33,20,26,102,87,80,103,92,35,25,26,53,37,22,27]]));

R:= NewPolyRing(QQ, SymbolRange("x",1,32 ), M, 1); 

use R;

A:=mat([

[0,0,0,0,0,0,0,0,0,0,0,0,0,0,0,0,1,1,1,1,1,1,1,1,1,1,1,1,1,1,1,1],
[0,0,0,0,0,0,0,0,1,1,1,1,1,1,1,1,1,1,1,1,0,0,0,0,0,0,0,0,1,1,1,1],
[0,0,0,0,1,1,1,1,1,1,1,1,0,0,0,0,0,0,0,0,1,1,1,1,0,0,0,0,1,1,1,1],
[0,0,0,0,1,1,1,1,0,0,0,0,1,1,1,1,0,0,0,0,0,0,0,0,1,1,1,1,1,1,1,1],
[0,0,1,1,0,0,1,1,1,1,0,0,1,1,0,0,0,0,1,1,1,1,0,0,1,1,0,0,0,0,1,1],
[0,1,1,0,0,1,1,0,1,0,0,1,1,0,0,1,0,1,1,0,1,0,0,1,1,0,0,1,0,1,1,0],
[0,1,0,1,0,1,0,1,1,0,1,0,1,0,1,0,0,1,0,1,1,0,1,0,1,0,1,0,0,1,0,1],
[1,1,1,1,1,1,1,1,1,1,1,1,1,1,1,1,1,1,1,1,1,1,1,1,1,1,1,1,1,1,1,1]
]);

ReducedGBasis( toric(A) );
}
\end{screen}
to CoCoA, then we can obtain the reduced Gr\"{o}bner basis in several seconds. 
The monomial order $\bm{w}$ is neither lexicographic nor reverse lexicographic. 
In fact, 
all monomial orders for which the reduced Gr\"{o}bner bases of $I_{G}$ consist of quadratic binomials 
are neither lexicographic nor reverse lexicographic.

\begin{Theorem}
Let $G$ be the graph of Fig.~$\ref{6-1}$.
Then $I_G$ has quadratic Gr\"{o}bner bases,
none of which are either lexicographic or reverse lexicographic.
In particular, $I_G$ is a counterexample to Conjecture~$\ref{conj}$.
\end{Theorem}

\begin{proof}
Since $G$ has an induced cycle of length 5,
$I_G$ has no reverse lexicographic quadratic Gr\"{o}bner bases by
Proposition~\ref{inducedcycle}.
Moreover, since $K_{2,3}$ is obtained by
contraction of an edge of $G$, 
$I_G$ has no lexicographic quadratic Gr\"{o}bner bases by
Proposition~\ref{k23contraction}.
\end{proof}

\section{Squarefree Veronese subrings and cut ideals of cycles}
If a graph $G$ is a cycle, then the cut ideal $I_G$ is generated by quadratic binomials by Proposition~\ref{quadgene}. 
Nagel-Petrovi\'{c} \cite[Proposition 3.2]{Petro2} claimed that the cut ideal of a cycle has a quadratic Gr\"{o}bner basis 
with respect to a lexicographic order. This claim relies on the following claims in Chifman-Petrovi\'{c} \cite{Petro1}:

\bigskip

\noindent
\underline{{\bf Claim 1} (\cite[Proposition 2]{Petro1})} \ Let $I_m$ be the toric ideal of phylogenetic invariants for the general group-based model on the
claw tree $K_{1,m}$} (defined later) which coincides with the cut ideals of the cycle of length $m+1$. Then $I_m$ is generated by $Q_m$  (defined later) which consists of quadratic binomials. 

\bigskip

\noindent
\underline{{\bf Claim 2} (\cite[Proposition 3]{Petro1})} \ The set $Q_m$ is a lexicographic Gr\"{o}bner basis of $I_m$ for any $m\ge 4$.

\bigskip

However, Claim 1 is not true for any $m\ge 5$. Therefore, Claim 2 is not true for any $m\ge 5$. 
Moreover, with respect to a lexicographic order given in \cite{Petro1}, the reduced Gr\"{o}bner basis of $I_m$ is not quadratic for any $m\ge 5$.   
In this section, we point out an error in the proof of 
\cite[Propositions 2 and 3]{Petro1} 
 for the cut ideal of the cycle and
present a lexicographic order for which 
the reduced Gr\"{o}bner basis of the cut ideal of the cycle of length $7$ consists of quadratic binomials.

First, we explain an error in the proof of \cite[Propositions 2 and 3]{Petro1}.
For each $m$-dimensional ($0,1$) vector $(i_{1},i_{2},\ldots ,i_{m})$,
we associate a variable $q_{i_{1}i_{2}\cdots i_{m}}$.
Let $K[q_{i_{1}i_{2}\ldots i_{m}}\ |\  i_{1},i_{2},\ldots ,i_{m}\in \{0,1\}]$
and $K[a_{i_{j}}^{(j)}\ | \ i_{j}\in \{0,1\} , j=1,\ldots ,m+1]$ be
polynomial rings over $K$.
Let
$$\varphi_{m}:
K[q_{i_{1}i_{2}\ldots i_{m}}\ |\  i_{1},i_{2},\ldots ,i_{m}\in \{0,1\}]\rightarrow K[a_{i_{j}}^{(j)}\ | \ i_{j}\in \{0,1\} , j=1,\ldots ,m+1]$$
be a homomorphism such that 
$\varphi _{m}(q_{i_{1}i_{2}\ldots i_{m}})=a_{i_{1}}^{(1)}a_{i_{2}}^{(2)}\ldots a_{i_{m}}^{(m)}a_{i_{1}+i_{2}+\cdots +i_{m} ({\rm mod} \ 2)}^{(m+1)}$ and let  $I_{m}$ be the kernel of $\varphi_{m}$. 
According to \cite{Petro2}, the ideal $I_{m}$ is the cut ideal of the cycle of length $m+1$. 
Let {$Q_{m}$} be a set of all quadratic 
binomials 
$$q_{i_{1}i_{2}\cdots i_{m}}q_{j_{1}j_{2}\cdots j_{m}}-q_{k_{1}k_{2}\cdots k_{m}}q_{l_{1}l_{2}\cdots l_{m}}\in I_{m}$$
satisfying one of the 
following properties:
\begin{enumerate}
\item For some $1\leq{a}\leq{m}$ and $j\in \{0,1\}$, 
\begin{equation*}
i_{a}=j_{a}=j=k_{a}=l_{a}
\end{equation*}
and the binomial 
$$q_{i_{1}\ldots i_{a-1}i_{a+1}\ldots i_{m}}q_{j_{1}\ldots j_{a-1}j_{a+1}\ldots j_{m}}-q_{k_{1}\ldots k_{a-1}k_{a+1}\ldots k_{m}}q_{l_{1}\ldots l_{a-1}l_{a+1}\ldots l_{m}}$$
belongs to $I_{m-1}$;
\item For each $1\leq{b}\leq{m}$, 
\begin{equation*}
i_{b}+j_{b}=1=k_{b}+l_{b}
\end{equation*}
and the binomial 
$$q_{i_{1}\ldots i_{b-1}i_{b+1}\ldots i_{m}}q_{j_{1}\ldots j_{b-1}j_{b+1}\ldots j_{m}}-q_{k_{1}\ldots k_{b-1}k_{b+1}\ldots k_{m}}q_{l_{1}\ldots l_{b-1}l_{b+1}\ldots l_{m}}$$
belongs to $I_{m-1}$.
\end{enumerate}
In \cite[Proposition 2]{Petro1},
$I_{m}$ is claimed to be generated by {$Q_{m}$} for any $m \ge 4$.
However, this is incorrect for $m \ge 5$.
Now, we consider the quadratic binomial
\begin{equation*}
q=q_{10101}q_{01010}-q_{11111}q_{00000}
\end{equation*}
 and the binomial $q'=q_{1101}q_{0010}-q_{1111}q_{0000}$.
Since $$\varphi_5(q_{10101}q_{01010})=\varphi_5(q_{11111}q_{00000})=a_{0}^{(1)}a_{0}^{(2)}a_{0}^{(3)}a_{0}^{(4)}a_{0}^{(5)}a_{0}^{(6)}
a_{1}^{(1)}a_{1}^{(2)}a_{1}^{(3)}a_{1}^{(4)}a_{1}^{(5)}a_{1}^{(6)},$$ $q$ belongs to $I_5$. On the other hand, since
$$\varphi_4(q_{1101}q_{0010})=a_{0}^{(1)}a_{0}^{(2)}a_{0}^{(3)}a_{0}^{(4)}a_{1}^{(1)}a_{1}^{(2)}a_{1}^{(3)}a_{1}^{(4)}(a_{1}^{(5)})^2, $$
$$\varphi_4(q_{1111}q_{0000})=a_{0}^{(1)}a_{0}^{(2)}a_{0}^{(3)}a_{0}^{(4)}a_{1}^{(1)}a_{1}^{(2)}a_{1}^{(3)}a_{1}^{(4)}(a_{0}^{(5)})^2,$$
$q'$ does not belong to $I_4$. Hence $q$ does not belong to $Q_5$. The following proposition shows that $q$ is not generated by $Q_5$.  
\begin{Proposition}
Let
\begin{equation*}
P=\left \{q_{i_{1}i_{2}i_{3}i_{4}i_{5}}q_{j_{1}j_{2}j_{3}j_{4}j_{5}}\  \middle |
\begin{array}{l}
i_{k}+j_{k}=1, 
i_{k},j_{k}\in \{0,1\} \mbox{ for }
1\leq{k}\leq{5}
\end{array} \right \}.
\end{equation*} 
Then any nonzero binomial $q=u-v$ where $u,v\in P$ does not belong to $Q_{5}$.
\end{Proposition}

\begin{proof}
Let 
$$q=
q_{i_{1}i_{2}i_{3}i_{4}i_{5}}q_{j_{1}j_{2}j_{3}j_{4}j_{5}}
-
q_{i_{1}'i_{2}'i_{3}'i_{4}'i_{5}'}q_{j_{1}'j_{2}'j_{3}'j_{4}'j_{5}'}
$$
be a nonzero binomial
where
$i_{k}+j_{k}=i_{k}'+j_{k}'=1$
and 
$i_{k},j_{k},i_{k}',j_{k}'\in \{0,1\}$
for
$1\leq{k}\leq{5}$.
It is trivial that $q$ does not satisfy property (1).
Since $i_{k}+j_{k}=i_{k}'+j_{k}'=1$ for $1\leq{k}\leq{5}$,
we have $$\sum_{k=1}^5 i_k+\sum_{k=1}^5 j_k = \sum_{k=1}^5 i_k'+\sum_{k=1}^5 j_k' =5.$$
Hence we may assume that 
$\sum_{k=1}^5 i_k \equiv
\sum_{k=1}^5 i_k' \equiv 1$
and 
$\sum_{k=1}^5 j_k \equiv\sum_{k=1}^5 j_k'\equiv0$
modulo $2$.
Since $q$ is not zero, $q_{i_{1}i_{2}i_{3}i_{4}i_{5}}
\neq q_{i_{1}'i_{2}'i_{3}'i_{4}'i_{5}'}$.
Thus we may assume that $i_k = 1$ and $i_k'=0$
for some $1 \le k \le 5$ (by exchanging $q_{i_{1}i_{2}i_{3}i_{4}i_{5}}$ and $q_{i_{1}'i_{2}'i_{3}'i_{4}'i_{5}'}$ if we need).
Then $j_k = 0$ and $j_k'=1$.
For example, if $k=1$, then
$$
q'=
q_{i_{2}i_{3}i_{4}i_{5}}q_{j_{2}j_{3}j_{4}j_{5}}
-
q_{i_{2}'i_{3}'i_{4}'i_{5}'}q_{j_{2}'j_{3}'j_{4}'j_{5}'}
$$
does not belong to $I_4$ since
$i_{2}+i_{3}+i_{4}+i_{5} \equiv j_{2}+j_{3}+j_{4}+j_{5} \equiv 0$ and 
$i_{2}'+i_{3}'+i_{4}'+i_{5}'
\equiv
j_{2}'+j_{3}'+j_{4}'+j_{5}'
\equiv 1$.
Thus $q$ does not satisfy property (2).
\end{proof}

Thus, $I_5$ is not generated by {$Q_5$}.
This is the error in the proof of \cite[Proposition 2]{Petro1}.
By this error, instead of $I_m$, an ideal that is strictly smaller than $I_m$ is considered in
the proof of \cite[Proposition 3]{Petro1}.
Unfortunately, with respect to 
a lexicographic order considered in \cite[Proposition 3]{Petro1}, the reduced Gr\"{o}bner basis of $I_m$ is not quadratic for $m \ge 5$. 
The computation for $m=5$ is given in 
\begin{center}
{\tt https:\slash\slash sci-tech.ksc.kwansei.ac.jp/\textasciitilde hohsugi/R\_Sakamoto/code\_cutideal}
\end{center}
We describe the number of binomials in the reduced Gr\"{o}bner basis with respect to a lexicographic order in \cite{Petro1} for $m=5$ in Table 1.
\begin{table}[h]
\begin{center}
\begin{tabular}{|c|c|}\hline
degree & the number of binomials \\ \hline
2 & 195 \\
3 & 10 \\
4 & 2 \\ \hline
\end{tabular}

\caption{The number of binomials in the reduced Gr\"{o}bner basis of $I_5$.}
\end{center}
\end{table}


Thus the existence of a quadratic Gr\"{o}bner basis of the cut ideal of a cycle
is now an open problem.
However, we will show that there exists a lexicographic order such that 
the reduced Gr\"{o}bner basis of the cut ideal of a cycle of length $7$
 consists of quadratic binomials.
In general, if $G$ is a cycle of length $m$, then it is known that $\{\delta_{C} \ | \ C\subset V(G)\}=\{(d_1,\ldots, d_m)
\in \{0,1\}^m \ | \ d_1+\cdots+d_m$ is even\}. 
Let $G$ be the cycle of length 7.
Then we have
\begin{center}
{$A_{G}=
\begin{pmatrix}
\bm{0} & &A& & &B& & &C& \\
1 & 1 & \cdots &1 & 1 & \cdots & 1& 1 &\cdots & 1
\end{pmatrix}
$},
\end{center}
where
\begin{center}
{$A=
\begin{pmatrix}
1 & 1 & 1 & 1 & 1 & 1 & 0 & 0 & 0 & 0 & 0 & 0 & 0 & 0 & 0 & 0 & 0 & 0 & 0 & 0 & 0 \\
1 & 0 & 0 & 0 & 0 & 0 & 1 & 1 & 1 & 1 & 1 & 0 & 0 & 0 & 0 & 0 & 0 & 0 & 0 & 0 & 0 \\
0 & 1 & 0 & 0 & 0 & 0 & 1 & 0 & 0 & 0 & 0 & 1 & 1 & 1 & 1 & 0 & 0 & 0 & 0 & 0 & 0 \\
0 & 0 & 1 & 0 & 0 & 0 & 0 & 1 & 0 & 0 & 0 & 1 & 0 & 0 & 0 & 1 & 1 & 1 & 0 & 0 & 0 \\
0 & 0 & 0 & 1 & 0 & 0 & 0 & 0 & 1 & 0 & 0 & 0 & 1 & 0 & 0 & 1 & 0 & 0 & 1 & 1 & 0 \\
0 & 0 & 0 & 0 & 1 & 0 & 0 & 0 & 0 & 1 & 0 & 0 & 0 & 1 & 0 & 0 & 1 & 0 & 1 & 0 & 1 \\
0 & 0 & 0 & 0 & 0 & 1 & 0 & 0 & 0 & 0 & 1 & 0 & 0 & 0 & 1 & 0 & 0 & 1 & 0 & 1 & 1
\end{pmatrix}
$},
\end{center}
\vspace{5truept}
\begin{center}
{$B=\left(
\begin{smallmatrix}
1 & 1 & 1 & 1 & 1 & 1 & 1 & 1 & 1 & 1 & 1 & 1 & 1 & 1 & 1 & 1 & 1 & 1 & 1 & 1 & 0 & 0 & 0 & 0 & 0 & 0 & 0 & 0 & 0 & 0 & 0 & 0 & 0 & 0 & 0 \\
1 & 1 & 1 & 1 & 1 & 1 & 1 & 1 & 1 & 1 & 0 & 0 & 0 & 0 & 0 & 0 & 0 & 0 & 0 & 0 & 1 & 1 & 1 & 1 & 1 & 1 & 1 & 1 & 1 & 1 & 0 & 0 & 0 & 0 & 0 \\
1 & 1 & 1 & 1 & 0 & 0 & 0 & 0 & 0 & 0 & 1 & 1 & 1 & 1 & 1 & 1 & 0 & 0 & 0 & 0 & 1 & 1 & 1 & 1 & 1 & 1 & 0 & 0 & 0 & 0 & 1 & 1 & 1 & 1 & 0 \\
1 & 0 & 0 & 0 & 1 & 1 & 1 & 0 & 0 & 0 & 1 & 1 & 1 & 0 & 0 & 0 & 1 & 1 & 1 & 0 & 1 & 1 & 1 & 0 & 0 & 0 & 1 & 1 & 1 & 0 & 1 & 1 & 1 & 0 & 1 \\
0 & 1 & 0 & 0 & 1 & 0 & 0 & 1 & 1 & 0 & 1 & 0 & 0 & 1 & 1 & 0 & 1 & 1 & 0 & 1 & 1 & 0 & 0 & 1 & 1 & 0 & 1 & 1 & 0 & 1 & 1 & 1 & 0 & 1 & 1 \\
0 & 0 & 1 & 0 & 0 & 1 & 0 & 1 & 0 & 1 & 0 & 1 & 0 & 1 & 0 & 1 & 1 & 0 & 1 & 1 & 0 & 1 & 0 & 1 & 0 & 1 & 1 & 0 & 1 & 1 & 1 & 0 & 1 & 1 & 1 \\
0 & 0 & 0 & 1 & 0 & 0 & 1 & 0 & 1 & 1 & 0 & 0 & 1 & 0 & 1 & 1 & 0 & 1 & 1 & 1 & 0 & 0 & 1 & 0 & 1 & 1 & 0 & 1 & 1 & 1 & 0 & 1 & 1 & 1 & 1
\end{smallmatrix}\right)
$},
\end{center}
\vspace{5truept}
\begin{center}
{$C=
\begin{pmatrix}
1 & 1 & 1 & 1 & 1 & 1 & 0 \\
1 & 1 & 1 & 1 & 1 & 0 & 1 \\
1 & 1 & 1 & 1 & 0 & 1 & 1 \\
1 & 1 & 1 & 0 & 1 & 1 & 1 \\
1 & 1 & 0 & 1 & 1 & 1 & 1 \\
1 & 0 & 1 & 1 & 1 & 1 & 1 \\
0 & 1 & 1 & 1 & 1 & 1 & 1
\end{pmatrix}$}.
\end{center}
%
In general, the configuration of the $(m,r)$-squarefree Veronese subring
is the configuration whose columns are
$$
\{(d_1,\ldots, d_m)
\in \{0,1\}^m \ | \ d_1+\cdots+d_m=r\}
.$$
The matrix $A$ is a configuration of the $(7,2)$-squarefree Veronese subring, and $B$ is a configuration of the $(7,4)$-squarefree Veronese subring. 
According to \cite[Theorem~$1.4$]{OhsugiHibi?}, 
there is a lexicographic order such that the reduced Gr\"{o}bner basis of the toric ideal of  the $(m,2)$-squarefree Veronese subring
consists of quadratic binomials for any integer $m \ge 2$. 
However, it is not known whether there is a lexicographic order such that 
the reduced Gr\"{o}bner basis of $I_{B}$ consists of quadratic binomials. Now, we consider the following question:

\begin{Question}
If we use lexicographic orders such that the reduced Gr\"{o}bner bases of $I_A$ and $I_B$ consist of quadratic binomials,
 do we obtain a lexicographic order such that the reduced Gr\"{o}bner basis of $I_{A_{G}}$ consists of quadratic binomials?
\end{Question}

To answer this question, we look for a lexicographic order $>_{1}$ such that the reduced Gr\"{o}bner basis of $I_{B}\subset K[y_{1},y_{2},\ldots ,y_{35}]$ consists of quadratic binomials.
For $i = 1,2,\ldots, 7$, we consider the subconfiguration $B_i$ of $B$ with column vectors consisting of  all column vectors of $B$ whose $i$-th component is one.
We consider combining lexicographic orders such that the reduced Gr\"{o}bner bases of $I_{B_i}$ consist of quadratic binomials.
We write down the lexicographic order $>_{1}$:\vspace{5truept} \\
$y_{1}>y_{2}>y_{4}>y_{3}>y_{5}>y_{7}>y_{6}>y_{10}>y_{9}>y_{8}>y_{11}>y_{13}>y_{12}>y_{16}>y_{15}>y_{14}>y_{20}>
y_{19}>y_{18}>y_{17}>y_{21}>y_{23}>y_{22}>y_{26}>y_{25}>y_{24}>y_{30}>y_{29}>y_{28}>y_{27}>y_{35}>y_{34}>y_{33}>
y_{32}>y_{31}$.\vspace{5truept} \\
Next, we consider combining two lexicographic orders such that the reduced Gr\"{o}bner bases of $I_{A}$ and $I_{B}$ consist
of quadratic binomials. 
We fix the order\\
 $x_{23}>x_{24}>x_{26}>x_{25}>x_{27}>x_{29}>x_{28}>x_{32}>x_{31}>
x_{30}>x_{33}>x_{35}>x_{34}>x_{38}>x_{37}>x_{36}>x_{42}>x_{41}>x_{40}>x_{39}>x_{43}>x_{45}>x_{44}>x_{48}>x_{47}>x_{46}>
x_{52}>x_{51}>x_{50}>x_{49}>x_{57}>x_{56}>x_{55}>x_{54}>x_{53}$\\
which corresponds to the lexicographic order $>_{1}$
and look for the order such that 
the reduced Gr\"{o}bner basis of $I_{A_{G}}$ consists of quadratic binomials
by modifying the order for $I_A$ using computational experiments.
A desired lexicographic order is\\
$x_{1}>x_{17}>x_{18}>x_{19}>x_{22}>x_{20}>x_{21}>x_{13}>x_{14}>x_{15}>x_{16}>x_{2}>x_{3}>x_{4}>x_{5}>x_{6}>x_{7}>
x_{8}>x_{9}>x_{10}>x_{11}>x_{12}>x_{23}>x_{24}>x_{26}>x_{25}>x_{27}>x_{29}>x_{28}>x_{32}>x_{31}>x_{30}>
x_{33}>x_{35}>x_{34}>x_{38}>x_{37}>x_{36}>x_{42}>x_{41}>x_{40}>x_{39}>x_{43}>x_{45}>x_{44}>x_{48}>x_{47}>x_{46}>
x_{52}>x_{51}>x_{50}>x_{49}>x_{57}>x_{56}>x_{55}>x_{54}>x_{53}>x_{58}>x_{59}>x_{60}>x_{61}>x_{62}>x_{63}>x_{64}$.\vspace{5truept} \\
The reduced Gr\"{o}bner basis of $I_{G}$ consists of $1050$ quadratic binomials. 
The computation is given in 
\begin{center}
{\tt https:\slash\slash sci-tech.ksc.kwansei.ac.jp/\textasciitilde hohsugi/R\_Sakamoto/code\_cutideal}
\end{center} 
Note that any cycle of length $\le 6$ is obtained by the sequence of contractions
from $G$.
Thus, we have the following.
\begin{Theorem}
 Let $G$ be a cycle of length $\le 7$.
 Then $I_G$ has a lexicographic quadratic Gr\"{o}bner basis.
 \end{Theorem}
\subsection*{Acknowledgment}
The author is grateful to the anonymous referees for their careful
reading and helpful comments.

\end{document}